\documentclass[a4paper, 11pt]{amsart}
\usepackage{amsmath,amssymb}
\usepackage{caption}
\usepackage{here}
\usepackage{color}
\usepackage{cite}
\usepackage{booktabs} 
\usepackage{graphicx}
\newtheorem{thm}{Theorem}[section]
\newtheorem{lem}[thm]{Lemma}
\newtheorem{prp}[thm]{Proposition}

\theoremstyle{definition}
\newtheorem{problem}[thm]{Problem}

\newcommand{\al}{\alpha}
\newcommand{\be}{\beta}

\newcommand{\ds}{\displaystyle}

\newcommand{\ch}[1]{\textcolor{black}{#1}}

\title[Reciprocal sum of  Fibonacci numbers
]
{
Continuous approximation to the reciprocal sum of the cubes of Fibonacci numbers
}

\author[Wontae Hwang, Jong-Do Park, Kyunghwan Song]
{Wontae Hwang$^1$, Jong-Do Park$^2$, Kyunghwan Song$^3$}


\address{Wontae Hwang, 
\ch{Department of Mathematics, and Institute of Pure and Applied Mathematics, Jeonbuk National University, Baekje-daero, Deokjin-gu, Jeonju-si, Jeollabuk-do, 54896 Republic of Korea}}

\email{hwangwon@jbnu.ac.kr}

\address{Jong-Do Park, Department of Mathematics and Research Institute for Basic Sciences, Kyung Hee University, Seoul 130-701, Korea}

\email{mathjdpark@khu.ac.kr}

\address{Kyunghwan Song, 
\ch{Department of Mathematics, Jeju National University, 102 Jejudaehak-ro Jeju, 63243, Republic of Korea}}

\email{\ch{khsong@jejunu.ac.kr}}


\begin{document}


\begin{abstract}
{Let $f_n$ be the $n$-th Fibonacci number with $f_1=f_2=1$.
Recently exact formulas for the integer parts of the tails of inverse reciprocal Fibonacci numbers 
have been obtained in several cases.
However, the cubic case ($s=3$) is much more complicated because of highly oscillating error terms. 
Thus it is difficult to construct a precise continuous approximation and algebraic estimates simultaneously. 
In this paper, we give a complete and unified algebraic method to solve this difficulty. 
More precisely we construct an explicit closed form of sequence $g_n$, preserving the principal part 
$f_n^3-f_{n-1}^3$, such that 
$\ds
\lim_{n\rightarrow\infty}\left\{
\left(
\sum^\infty_{k=n}\frac{1}{f_k^3}
\right)^{-1}-g_n
\right\}=0.
$
By decomposing the error terms and investigating the algebraic identities, we establish the lower and upper bounds
$
\ds g_n<\left(
\sum^\infty_{k=n}\frac{1}{f_k^3}
\right)^{-1}<g_n+2/f_n
$
for sufficiently large $n$. 
As an application of the explicit form of the sequence $g_n$ and these estimates, we completely determine the exact value of the floor function for $s=3$.
}
\end{abstract}

\allowdisplaybreaks

\maketitle

\section{\bf Introduction}

Let $f_n$ be the $n$-th Fibonacci number defined by the recurrence relation
$$
f_0=0,~ f_1=1,\mbox{ and }f_{n+1}=f_n+f_{n-1}\mbox{ for any }n\in\mathbb{N}.
$$

Recently, the research on the tailing of the convergent series has attracted significant interest.
In particular, we focus our attention on the reciprocal sum of Fibonacci numbers and related sequences \cite{CC18, LP, LP2, LW, WZ15, Zhang11, ZW12}. 
By Binet's formula \cite{Ko}, $f_n$ can be written as the explicit form $f_n=\frac{\al^n-\be^n}{\al-\be}$ for any integers $n$, where $\al=\frac{1+\sqrt{5}}{2}$ and $\be=\frac{1-\sqrt{5}}{2}$ 
are solutions of the characteristic equation $x^2=x+1$.
Since $f_n$ is asymptotically comparable to $\al^n$ for sufficiently large $n$, the infinite series
$$
F(s)=\sum^\infty_{k=1}\frac{1}{f_k^s}
$$
converges for all $s\geq1$.
By the Cauchy criterion, the tail of the convergent series $F(s)$ tends to zero, so that
$$
\lim_{n\rightarrow\infty}\left(
\sum^\infty_{k=n}\frac{1}{f_k^s}
\right)^{-1}=\infty.
$$

Motivated by this observation,  Ohtsuka and Nakamura \cite{ON} initiated the study of the integer parts of these inverse tails. 
They determined the exact values of the floor function 
for the cases $s=1$ and $s=2$. For instance, if $s=1$, then
\begin{align*}
\left\lfloor \left(\sum^\infty_{k=n}\frac{1}{f_k}\right)^{-1}\right\rfloor&=
\begin{cases}
f_{n-2}, & n\geq2 \mbox{ is even};\\
f_{n-2}-1, & n\geq1 \mbox{ is odd}.
\end{cases}
\end{align*}
To understand a deeper continuous behavior 
underlying these floor values, 
subsequent studies \cite{LP, LP2} introduced an 
asymptotic approach.
Instead of bounding the sum between integers, 
they have tried to find a precise continuous sequence 
$g_{s,n}$ that approaches the inverse reciprocal sum asymptotically. 
This relationship is formalized as follows:
\vskip1pc

\begin{problem}
Find a precise representation $g_{s,n}$ such that 
$$
\lim_{n\rightarrow\infty}
\left\{\left(
\sum^\infty_{k=n}\frac{1}{f_k^s}
\right)^{-1}-g_{s,n}
\right\}=0.
$$
\end{problem}

Solving this problem not only reveals the exact algebraic structure of the infinite sum 
but also gives the exact value of its floor function as an application. 
For the cases of $s=1,2$, it has already been proved in \cite{LP,LP2} that
\begin{align*}
&g_{1,n}=f_n-f_{n-1},\\
&g_{2,n}=f_n^2-f_{n-1}^2+\frac{2}{3}(-1)^n.
\end{align*}

Moreover, for $s=4$, we proved in \cite{HPS} that 
$$g_{4,n}=f_n^4-f_{n-1}^4+\frac{2(-1)^n}{5}f_{2n-1}+\frac{2\sqrt{5}}{75}.$$

The cubic case $(s=3)$ has historically been more complicated and was independently 
approached from two distinct viewpoints. Wang \cite{W12} determined the floor function for $s=3$ 
by proving the inequalities. 
However, their formula still contained the floor operation with modulo $11$ fractions. 
Furthermore, their algebraic form does not contain the principal term of the form $f_n^s-f_{n-1}^s$
and they did not provide the explicit continuous sequence $g_{3,n}$ whose error term approaches zero. 
On the other hand, a recent study in \cite{LYY} investigated the asymptotic behavior of the generalized Fibonacci numbers and derived an approximation formula that satisfies the limit definition for $s=3$. 
In fact, their method only used Binet's formula without constructing algebraic bounds, leaving 
the exact evaluation of the floor function unsolved.

The main goal of this paper is to bridge this gap by proving a complete and unified algebraic framework. 
We construct the exact continuous sequence $g_{3,n}$ containing the natural principal part $f_n^3-f_{n-1}^3$ and construct algebraic inequalities. 
By explaining the fractional remainders through modulo-5 method, we
solve both the continuous asymptotic limit and the discrete floor function value as a new form.

The main theorem of this paper is stated as follows:

\vskip1pc

{\bf{Theorem \ref{thm-main}.}}  
For sufficiently large $n$, we have
$$
g_{3,n}<\left(\sum^\infty_{k=n}\frac{1}{f_k^3}\right)^{-1}<g_{3,n}+\frac{2}{f_n},
$$
where
$$
g_{3,n}:=f_n^3-f_{n-1}^3+\frac{3(-1)^n}{11}(3f_n-f_{n-1}).
$$

\vskip1pc

As an application of Theorem \ref{thm-main}, 
we compute the exact value of its floor function.

\vskip1pc

{\bf{Theorem \ref{thm-floor}.}} 
The exact value of the floor function is given by
$$
\left\lfloor
\left(
\sum^\infty_{k=n}\frac{1}{f_k^3}
\right)^{-1}
\right\rfloor
=
\begin{cases}
g_n-w_n & n\not\equiv3\pmod{5},\\
g_n-w_n-1 & n\equiv3\pmod{5}.
\end{cases}
$$

For the sake of simplicity, we will write $g_n=g_{3,n}$ throughout the remainder of this paper.

\vskip1pc

This paper is organized as follows.
In Section 2, we briefly review important identities of the Fibonacci sequences. 
In Section 3, we investigate the oscillating error terms and 
construct the correction sequences $u_n$, $w_n$, and $v_n$ based on their modulo-5 periodic phenomenon. 
In Section 4, we present our main theorem 
and demonstrate how the explicit approximation $g_n=g_{3,n}$ gives the exact value of its floor function
(See Theorem \ref{thm-floor}). 
In Section 5, we provide the algebraic proofs for the lower and upper estimates (See Theorem \ref{thm-main}).
Finally in Section 6, we end with concluding remarks.

\vskip1pc

\section{\bf Basic identities}

In this section, we briefly review several fundamental properties and identities involving the Fibonacci 
and Lucas sequence \cite{Ko}.
These identities will be extensively used as standard 
algebraic tools \cite{Rab96, Yang18, Yayenie11} for evaluating finite sums and simplifying high-degree cross terms in our proofs. 

Recall that the Lucas number $L_n$ satisfies the same recurrence relation as the Fibonacci number, namely 
$L_{n+1}=L_n+L_{n-1}$, but with the different initial conditions $L_0=2$ and $L_1=1$. By Binet's formula, 
the $n$-th Fibonacci and Lucas numbers can be explicitly expressed as
\begin{align}\label{defn}
f_n=\frac{\al^n-\be^n}{\sqrt{5}} \mbox{ and }
L_n=\al^n+\be^n.
\end{align}
Using these closed forms, one can naturally extend 
$f_n$ and $L_n$ to negative indices, which will be used when we compute the error terms.

\begin{lem}\label{lem-negative}
For any integer $n$, we have
$$
f_{-n}=(-1)^{n+1}f_n\mbox{ and }
L_{-n}=(-1)^nL_n.
$$
\end{lem}

Using this property, we can obtain the values for negative indices such as $L_{-5}=-L_5=-11$, which will be used in the proof of Proposition \ref{prp-vn}.
The following lemma can be obtained using a general generating function for Fibonacci sequence. 

\begin{lem}\label{lem-gen}
For any integers $a,b$ and a real number $x$ such that $1-xL_a+x^2(-1)^a\neq0$, we have
\begin{align*}
\sum^k_{r=0}f_{ar+b}x^r=\frac{f_b-x(-1)^af_{b-a}-x^{k+1}f_{a(k+1)+b}+x^{k+2}(-1)^af_{ak+b}}{1-xL_a+x^2(-1)^a},
\end{align*}
where $L_a$ denotes the $a$-th Lucas number. 
\end{lem}

\begin{proof}
By Binet's formula (\ref{defn}), we have 
\begin{align*}
\sum^k_{r=0}f_{ar+b}x^r
&=\frac{1}{\sqrt{5}}\sum^k_{r=0}(\al^{ar+b}-\be^{ar+b})x^r\\
&=\frac{1}{\sqrt{5}}\left(
\al^b\sum^k_{r=0}(\al^ax)^r-\be^b\sum^k_{r=0}(\be^ax)^r
\right)\\
&=\frac{1}{\sqrt{5}}\left(
\al^b\frac{1-(\al^ax)^{k+1}}{1-\al^ax}-
\be^b\frac{1-(\be^ax)^{k+1}}{1-\be^ax}
\right)\\
&=\frac{1}{\sqrt{5}}
\frac{\al^b(1-\al^{a(k+1)}x^{k+1})-\be^b(1-\be^{a(k+1)}x^{k+1})(1-\al^ax)}{(1-\al^ax)(1-\be^ax)}.
\end{align*}
Using the identity $\al\be=-1$, the numerator is 
\begin{align*}
&(\al^b-\be^b)-x(\al^b\be^a-\be^b\al^a)-x^{k+1}(\al^{a(k+1)+b}-\be^{a(k+1)+b})+x^{k+2}(\al^{ak+a+b}\be^a-\be^{ak+a+b}\al^a)\\
=&(\al^b-\be^b)-x(-1)^a(\al^{b-a}-\be^{b-a})-x^{k+1}(\al^{a(k+1)+b}-\be^{a(k+1)+b})
+x^{k+2}(-1)^a(\al^{ak+b}-\be^{ak+b})
\end{align*}
and the denominator is 
\begin{align*}
(1-\al^ax)(1-\be^ax)=1-(\al^a+\be^a)x+(\al\be)^ax^2=1-xL_a+x^2(-1)^a,
\end{align*}
which completes the proof.
\end{proof}

In fact, we can prove Lemma \ref{lem-gen} also using the generating function 
\begin{align*}
\sum^\infty_{s=0}f_{as+b}x^s
=\frac{f_b-x(-1)^af_{b-a}}
{1-xL_a+x^2(-1)^a}.
\end{align*}
See the above formula in \cite{Ko}. Using the above generating formula, we have 
\begin{align*}
\sum^k_{r=0}f_{ar+b}x^r
&=\sum^\infty_{r=0}f_{ar+b}x^r-\sum^\infty_{r=k+1}f_{ar+b}x^r\\
&=\sum^\infty_{r=0}f_{ar+b}x^r-x^{k+1}\sum^\infty_{r=0}f_{ar+a(k+1)+b}x^r\\
&=\frac{f_b-x(-1)^af_{b-a}}{1-xL_a+x^2(-1)^a}-x^{k+1}
\left(
\frac{f_{a(k+1)+b}-x(-1)^af_{a(k+1)+b-a}}{1-xL_a+x^2(-1)^a}
\right).
\end{align*}
Note that the infinite series converges in the interval of convergence.
Through analytic continuation, the rational function is well-defined for any real $x$ 
provided that the denominator is non-zero.
\vskip1pc

The following lemma will be used when we reduce the higher order Fibonacci numbers into linear combinations of $f_n$ and $f_{n-1}$
(See \cite{Ko}).

\begin{lem}\label{lem-red}
For any integers $n,m$, the addition formula for Fibonacci numbers holds:
$$
f_{n+m}=f_{m+1}f_n+f_mf_{n-1}.
$$
\end{lem}

Finally, we will use the following Cassini's identity 
when we compute the polynomial identities involving  Fibonacci numbers (See \cite{Ko}).

\begin{lem}\label{lem-Cassini}
For any integers $n$, Cassini's identity holds:
$$
f_{n-1}f_{n+1}=f_n^2+(-1)^n.
$$
\end{lem}

\vskip1pc

\section{\bf{Observation}}

For any $n\in\mathbb{N}$, we consider the infinite reciprocal cubic sum $S_n$ defined by
$$
S_n=\sum^\infty_{k=n}\frac{1}{f_k^3}.
$$
Our goal in this paper is to determine an explicit form of sequence $g_n$ such that 
$$
S_n^{-1}-g_n\rightarrow0\mbox{ as } n\rightarrow\infty.
$$
From the patterns proved in previous results, it is reasonable to expect that 
$g_n$ involves the cubic difference $h_n$ given by
$$
h_n=f_n^3-f_{n-1}^3.
$$

\vskip1pc

Based on the asymptotic behavior of the Fibonacci numbers, it is natural to formulate 
an initial hypothesis that 
\begin{align}\label{hypo-1}
S_n^{-1}-h_n\rightarrow0\mbox{ as } n\rightarrow\infty.
\end{align}
In fact, by applying Binet's formula, $f_k\sim\frac{\al^k}{\sqrt{5}}$ as $k\rightarrow\infty$.
Then we have the approximation of $S_n$ as follows:
\begin{align*}
S_n=\sum^\infty_{k=n}\frac{1}{f_k^3}\sim\sum^\infty_{k=n}\frac{(\sqrt{5})^3}{\al^{3k}}
=\frac{(\sqrt{5})^3\al^{-3n}}{1-\al^{-3}}.
\end{align*}
It follows that
\begin{align*}
S_n^{-1}\sim\frac{\al^{3n}-\al^{3n-3}}{(\sqrt{5})^3}\sim f_n^3-f_{n-1}^3,
\end{align*}
which motivates the hypothesis (\ref{hypo-1}).

\vskip1pc

However, numerical analysis shows that this hypothesis (\ref{hypo-1}) is completely false. 
Instead of converging to zero, 
the difference $S_n^{-1}-h_n$ oscillates continuously and its absolute value grows. 
This divergence is explicitly demonstrated in Table \ref{table-un-wn}, 
which shows the values of $S_n^{-1}-h_n$.

\vskip1pc

\subsection{Oscillating differences}
To understand this phenomenon, 
we compute the values of $S_n^{-1}-h_n$ for the first few terms. 
By observing these values, we can decompose the error $S_n^{-1}-h_n$ into 
a dominant integer sequence $u_n$ and a smaller fractional remainder $w_n$ 
for small values of $n$.\\

\begin{table}[h]
\centering
\begin{tabular}{cccccc}
\toprule
$n$ & Mod 5 & $S_n^{-1}-h_n$ & $u_n$ & $w_n$ \\
\midrule
4 & 4 & $\approx1.909$ & $1$ & $10/11$\\
5 & 0 & $\approx-3.272$ & $-4$ & $8/11$\\
6 & 1 & $\approx5.181$ & $5$ & $2/11$\\
$7$ & 2 & $\approx-8.454$ & $-9$ & $6/11$\\
$8$ & $3$ & $\approx13.636$ & $14$ & $-4/11$\\
$9$ & $4$ & $\approx-22.090$ & $-23$ & $10/11$\\
\bottomrule
\end{tabular}
\caption{Investigation of $S_n^{-1}-h_n$}
\label{table-un-wn}
\end{table}

First, we derive an explicit formula for $u_n$.
The sequence of absolute values $|u_n|$ are $(1,4,5,9,14,23,\ldots)$, which obeys the standard Fibonacci addition rule. 
This fundamental property motivates the formal definition of $u_n$ by the second-order linear recurrence relation:
$$
u_{n+1}=u_{n-1}-u_n\mbox{ for all }n\geq5
$$
with the initial conditions $u_4=1$ and $u_5=-4$.
If we let $x_n=(-1)^nu_n$, then $x_n$ satisfies the standard recurrence 
$$
x_{n+1}=x_{n-1}+x_{n}
$$
with $x_4=1$ and $x_5=4$.
Since $x_n$ arises from the Fibonacci recurrence relation, 
it can be uniquely expressed as a linear combination of $f_{n}$ and $f_{n-1}$.
Thus there exist constants $A$ and $B$ such that $x_n=Af_n+Bf_{n-1}$. 
Using the initial values $x_4=1$ and $x_5=4$, we have $A=5$ and $B=-7$. 
It follows that $x_n=5f_n-7f_{n-1}$, so that 
\begin{align}\label{form-un}
u_n=(-1)^n(5f_n-7f_{n-1}).
\end{align}

Second, the fractional numbers $w_n$ exhibit a strict periodic pattern. 
Precisely, we have
$$
w_n=\frac{8}{11},\frac{2}{11},\frac{6}{11},-\frac{4}{11},\frac{10}{11}~~\mbox{ for }
n\equiv0,1,2,3,4\pmod{5}.
$$
Accordingly, we refine our hypothesis to expect that 
\begin{align}\label{hypo-2}
S_n^{-1}-h_n-u_n-w_n\rightarrow0\mbox{ as } n\rightarrow\infty.
\end{align}

\vskip1pc

\subsection{The Modulo-5 breakdown}

For $n\leq9$, the revised hypothesis (\ref{hypo-2}) holds. 
However, an unexpected structural breakdown occurs when $n=10$.
To investigate this phenomena, we examine the exact interval where 
it occurs by introducing the block index $k=\lfloor\frac{n}{5}\rfloor$.
Thus the range $10\leq n\leq 14$ corresponds to the block $k=2$.\\

\begin{table}[h]
\centering
\begin{tabular}{cccccc}
\toprule
$n$ & Mod 5 & $S_n^{-1}-h_n-u_n-w_n$ & $v_n$ & Fibonacci expression of $v_n$ \\
\midrule
9 & 4 & $\approx0.000$ & $0$ & $0$\\
\midrule
10 & 0 & $\approx-2.000$ & $-2$ & $-2f_1$\\
11 & 1 & $\approx2.000$ & $2$ & $2f_2$ \\
12 & 2 & $\approx-4.000$ & $-4$ & $-2f_3$\\
$13$ & 3 & $\approx6.000$ & $6$ & $2f_4$\\
$14$ & $4$ & $\approx-10.000$ & $-10$ & $-2f_5$\\
\midrule
$15$ & $0$ & $\approx14.000$ & $14$ & $2f_6-2f_1$\\
\bottomrule
\end{tabular}
\caption{Investigation of $S_n^{-1}-h_n-u_n-w_n$ for $k=2$}
\label{table-vn}
\end{table}

{\tiny{
\begin{table}[h]
\centering
\begin{tabular}{c|ccccc}
\toprule
Interval $k$ ($n=5k+i$) & $i=0$ & $i=1$ & $i=2$ & $i=3$ & $i=4$ \\
\midrule
$k=2$ ($10\leq n\leq 14$) & $-2f_1$ & $2f_2$ & $-2f_3$ & $2f_4$ & $-2f_5$\\
\addlinespace
$k=3$ ($15\leq n\leq 19$) 
& $2f_6-2f_1$ 
& $-2f_7+2f_2$ 
& $2f_8-2f_3$
& $-2f_9+2f_4$
& $2f_{10}-2f_5$
\\
\addlinespace
$k=4$ ($20\leq n\leq 24$) 
& $-2f_{11}+2f_6-2f_1$
& $2f_{12}-2f_7+2f_2$
& $-2f_{13}+2f_8-2f_3$
& $2f_{14}-2f_9+2f_4$
& $-2f_{15}+2f_{10}-2f_5$\\
\bottomrule
\end{tabular}
\caption{Module $5$ phenomena of $v_n$ for blocks}
\label{table-vn-2}
\end{table}
}}

\subsection{The Modulo-5 blocks} 
From Table \ref{table-vn-2}, we can formally define the formula of the secondary error term $v_n$ as the finite sum:
\begin{align}\label{form-vn}
v_n=\sum^{\lfloor n/5\rfloor}_{r=2}
(-2)(-1)^{n+r}f_{n-5r+1}.
\end{align}
In particular, we see that $v_n=0$ for $n\leq 9$ and 
$v_n=(-2)(-1)^nf_{n-9}$ for $10\leq n\leq 14$.

\vskip1pc

Taking all correction sequences into account, finally we expect that 
\begin{align}\label{hypo-3}
S_n^{-1}-h_n-u_n-v_n-w_n\rightarrow0\mbox{ as }n\rightarrow\infty.
\end{align}

Now we will prove that 
$$
g_n=h_n+u_n+v_n+w_n
$$
provides the exact algebraic approximation satisfying the problem suggested in Section 1.

\vskip1pc

\begin{prp}\label{prp-vn}
For any $n\geq10$, we have
$$
v_n=\frac{-2(-1)^n}{11}(23f_n-37f_{n-1})-w_n
$$
\end{prp}

\begin{proof}
Let $k=\lfloor n/5\rfloor$. Recall the definition of $v_n$:
\begin{align*}
v_n=(-2)(-1)^n\sum^k_{r=2}(-1)^rf_{n-5r+1}.
\end{align*}
Note that 
\begin{align*}
\sum^k_{r=2}(-1)^rf_{n-5r+1}=
\sum^k_{r=0}(-1)^rf_{n-5r+1}-(f_{n+1}-f_{n-4}).
\end{align*}

By Lemma \ref{lem-gen} with $a=-5$, $b=n+1$, and $x=-1$, we have
\begin{align*}
\sum^k_{r=0}(-1)^rf_{n-5r+1}
&=\frac{f_{n+1}-f_{n+6}+(-1)^kf_{n-5k-4}+(-1)^{k+1}f_{n-5k+1}}{1+L_{-5}+(-1)}\\
&=\frac{f_{n+1}-f_{n+6}+(-1)^kf_{n-5k-4}+(-1)^{k+1}f_{n-5k+1}}{-11}.
\end{align*}

Combining these evaluations gives the closed form
\begin{align}\label{form-sum}
\sum^k_{r=2}(-1)^rf_{n-5r+1}=
\frac{f_{n+6}-f_{n+1}}{11}
-(f_{n+1}-f_{n-4})
+(-1)^k\frac{f_{n-5k+1}-f_{n-5k-4}}{11}.
\end{align}

By expressing the general index as $n=5k+i$ for some $0\leq i\leq 4$, where $k=\lfloor n/5\rfloor$, 
the sum (\ref{form-sum}) simplifies
\begin{align*}
\sum^{\lfloor n/5\rfloor}_{r=2}(-1)^rf_{n-5r+1}=
\frac{f_{n+6}-12f_{n+1}+11f_{n-4}}{11}
+(-1)^k\frac{f_{i+1}-f_{i-4}}{11}.
\end{align*}
By Lemma \ref{lem-red}, we can reduce the numerator of the term into a linear combination of $f_n$ and $f_{n-1}$ as follows:
\begin{align*}
f_{n+6}-12f_{n+1}+11f_{n-4}
&=(13f_n+8f_{n-1})-12(f_n+f_{n-1})+11(2f_n-3f_{n-1})\\
&=23f_n-37f_{n-1}.
\end{align*}
Substituting these reduced formulas into the definition for $v_n$, we have
\begin{align*}
v_n
&=(-2)(-1)^n\left[
\frac{23f_n-37f_{n-1}}{11}
+(-1)^k\frac{f_{i+1}-f_{i-4}}{11}
\right]\\
&=\frac{(-2)(-1)^n}{11}(23f_n-37f_{n-1})
-\frac{2(-1)^{n+k}}{11}(f_{i+1}-f_{i-4}).
\end{align*}
Since $n+k=6k+i$, we have $(-1)^{n+k}=(-1)^i$. We now claim that the two remainder terms are exactly same as $w_n$:
\begin{align*}
w_n=w_i=\frac{2(-1)^i}{11}(f_{i+1}-f_{i-4}).
\end{align*}

To verify this claim, we evaluate this formula for each remainder $0\leq i\leq4$ as follows:
\begin{align*}
&w_0=\frac{2}{11}(f_1-f_{-4})=\frac{2}{11}(1-(-3))=\frac{8}{11},\\
&w_1=\frac{-2}{11}(f_2-f_{-3})=\frac{-2}{11}(1-2)=\frac{2}{11},\\
&w_2=\frac{2}{11}(f_3-f_{-2})=\frac{2}{11}(2-(-1))=\frac{6}{11},\\
&w_3=\frac{-2}{11}(f_4-f_{-1})=\frac{-2}{11}(3-1)=\frac{-4}{11},\\
&w_4=\frac{2}{11}(f_5-f_{0})=\frac{2}{11}(5-0)=\frac{10}{11}.
\end{align*}
Hence we conclude that
\begin{align*}
v_n=\frac{(-2)(-1)^n}{11}(23f_n-37f_{n-1})-w_n.
\end{align*}

\end{proof}

\begin{prp}
For any $n\geq10$, the unified error term $E_n$ is simplified to
$$
E_n=\frac{3(-1)^n}{11}(3f_n-f_{n-1}).
$$
\end{prp}
\begin{proof}
By the formula (\ref{form-un}) and Proposition \ref{prp-vn}, the fractional term $w_n$ is canceled, so that
\begin{align*}
E_n
&=u_n+(v_n+w_n)\\
&=(-1)^n(5f_n-7f_{n-1})+\frac{-2(-1)^n}{11}(23f_n-37f_{n-1})\\
&=\frac{3(-1)^n}{11}(3f_n-f_{n-1}).
\end{align*}
\end{proof}

\vskip 1pc

\section{\bf{Main results}}

We now present our main theorems for the reciprocal cubic sum $S_n$.

\begin{thm}\label{thm-main}
For all sufficiently large $n$, the inverse reciprocal sum $S_n^{-1}$ satisfies 
\begin{align*}
g_n<S_n^{-1}<g_n+c_n,
\end{align*}
where $c_n=2/f_n$. Thus we have
\begin{align*}
\lim_{n\rightarrow\infty}
\left\{
\left(
\sum^\infty_{k=n}\frac{1}{f_k^3}
\right)^{-1}-g_n
\right\}=0,
\end{align*}
so that $g_{3,n}=g_n$.
\end{thm}

\begin{thm}\label{thm-floor} 
For all $n\geq10$, the exact value of the floor function of $S_n^{-1}$ is given by
$$
\lfloor
S_n^{-1}
\rfloor
=
\begin{cases}
g_n-w_n & n\not\equiv3\pmod{5},\\
g_n-w_n-1 & n\equiv3\pmod{5}.
\end{cases}
$$
\end{thm}

\begin{proof}
By Theorem \ref{thm-main}, we have the inequality 
$$g_n<S_n^{-1}<g_n+\frac{2}{f_n}$$ for sufficiently large $n$. 
Recall that $g_n=h_n+u_n+v_n+w_n$. 
Since $h_n$, $u_n$, and $v_n$ are integers, we see that $g_n-w_n=h_n+u_n+v_n$ is an integer. Note that
\begin{align*}
(g_n-w_n)+w_n<S_n^{-1}<(g_n-w_n)+w_n+\frac{2}{f_n}.
\end{align*}
Recall that $w_n\in\{\frac{8}{11},\frac{2}{11},\frac{6}{11},\frac{-4}{11},\frac{10}{11}\}$
and for sufficiently large $n$ the term $\frac{2}{f_n}$ is less than $\frac{1}{11}$.

If $n\not\equiv3\pmod{5}$, then $w_n$ takes a positive value from the set 
$\{\frac{8}{11},\frac{2}{11},\frac{6}{11},\frac{10}{11}\}$. Then
\begin{align*}
g_n-w_n<(g_n-w_n)+w_n<S_n^{-1}<(g_n-w_n)+w_n+\frac{2}{f_n}<(g_n-w_n)+1,
\end{align*}
which proves $\lfloor S_n^{-1}\rfloor=g_n-w_n$.
\vskip1pc
If $n\equiv3\pmod{5}$, then $w_n=\frac{-4}{11}$. It follows that 
\begin{align*}
(g_n-w_n)-1<S_n^{-1}<g_n-w_n,
\end{align*}
which proves $\lfloor S_n^{-1}\rfloor=g_n-w_n-1$.
\end{proof}

\vskip1pc

\section{\bf{Proof of the Main Theorem}}

To prove the inequalities of the Main Theorem, we first introduce and prove a fundamental algebraic inequality.

\begin{prp}\label{prp-Delta}
For all integers $n\geq1$, let 
$\Delta_n:=(g_{n+1}-g_n)f_n^3-g_ng_{n+1}$. Then
$$\Delta_n>0.$$
\end{prp}

\begin{proof}[Proof of Proposition \ref{prp-Delta}]
Recall that $g_n$ is defined by the sum $g_n=h_n+E_n$, where the principal part is $h_n=f_n^3-f_{n-1}^3$ and the unified error is
$$
E_n=\frac{3(-1)^n}{11}(3f_n-f_{n-1}).
$$

Substituting $g_n=h_n+E_n$ into the algebraic definition of $\Delta_n$, we can 
naturally partition $\Delta_n$ into three blocks:
\begin{align*}
\Delta_n
&=
\left[
(h_{n+1}-h_n)f_n^3-h_nh_{n+1}
\right]
+\left[
(E_{n+1}-E_n)f_n^3-h_nE_{n+1}-h_{n+1}E_n
\right]
-E_nE_{n+1}\\
&=:X_n+Y_n+Z_n.
\end{align*}

We now evaluate each of these algebraic blocks explicitly. 
For the principal block $X_n$, we apply Lemma \ref{lem-Cassini} to obtain
\begin{align*}
X_n
&=(h_{n+1}-h_n)f_n^3-h_nh_{n+1}\\
&=(f_{n+1}^3-2f_n^3+f_{n-1}^3)f_n^3
-(f_n^3-f_{n-1}^3)(f_{n+1}^3-f_n^3)\\
&=-f_n^6+(f_{n-1}f_{n+1})^3\\
&=-f_n^6+(f_n^2+(-1)^n)^3\\
&=3(-1)^nf_n^4+3f_n^2+(-1)^n.
\end{align*}

Next, the cross-term block $Y_n$ is simplified to
\begin{align*}
Y_n&=(E_{n+1}-E_n)f_n^3-h_nE_{n+1}-h_{n+1}E_n\\
&=(E_{n+1}-E_n)f_n^3-(f_n^3-f_{n-1}^3)E_{n+1}-(f_{n+1}^3-f_n^3)E_{n}\\
&=E_{n+1}f_{n-1}^3-E_nf_{n+1}^3.
\end{align*}

Substituting the definition of $E_n$ and expressing $f_{n+1}=f_n+f_{n-1}$, we obtain
\begin{align*}
Y_n
&=E_{n+1}f_{n-1}^3-E_nf_{n+1}^3\\
&=\frac{3(-1)^{n+1}}{11}(3f_{n+1}-f_n)f_{n-1}^3-\frac{3(-1)^n}{11}(3f_n-f_{n-1})f_{n+1}^3\\
&=\frac{3(-1)^{n+1}}{11}(3f_{n+1}f_{n-1}^3-f_nf_{n-1}^3+3f_nf_{n+1}^3-f_{n-1}f_{n+1}^3)\\
&=\frac{3(-1)^{n+1}}{11}(3f_nf_{n-1}^3+3f_{n-1}^4-f_nf_{n-1}^3\\
&+3f_n^4+9f_n^3f_{n-1}+9f_n^2f_{n-1}^2+3f_nf_{n-1}^3
-f_{n-1}^4-3f_{n-1}^3f_n-3f_{n-1}^2f_n^2-f_{n-1}f_n^3)\\
&=\frac{3(-1)^{n+1}}{11}(3f_n^4+8f_n^3f_{n-1}+6f_n^2f_{n-1}^2+2f_nf_{n-1}^3+2f_{n-1}^4)
\end{align*}

To eliminate the higher powers of $f_{n-1}$, by Lemma \ref{lem-Cassini}, we have
\begin{align}\label{form-ele}
f_{n-1}^2=f_{n-1}(f_{n+1}-f_n)=f_n^2+(-1)^n-f_{n-1}f_n.
\end{align}

By substituting the formula (\ref{form-ele}) into the polynomial expansion, we have 
\begin{align*}
&3f_n^4+8f_n^3f_{n-1}+6f_n^2f_{n-1}^2+2f_nf_{n-1}^3+2f_{n-1}^4\\
&=3f_n^4+8f_n^3f_{n-1}+6f_n^2(f_n^2+(-1)^n-f_{n-1}f_n)\\
&+2f_nf_{n-1}(f_n^2+(-1)^n-f_{n-1}f_n)\\
&+2(f_n^4+1+f_{n-1}^2f_n^2+2(-1)^nf_n^2-2f_{n-1}f_n^3-2(-1)^nf_{n-1}f_n)\\
&=11f_n^4+10(-1)^nf_n^2-2(-1)^nf_nf_{n-1}+2.
\end{align*}

It follows that
\begin{align*}
Y_n=-3(-1)^nf_n^4-\frac{30}{11}f_n^2+\frac{6}{11}f_nf_{n-1}-\frac{6(-1)^n}{11}.
\end{align*}

Finally, we evaluate the block $Z_n$.
By Lemma \ref{lem-Cassini}, we have
\begin{align*}
Z_n
&=-E_nE_{n+1}\\
&=-\frac{3(-1)^n}{11}(3f_n-f_{n-1})\cdot\frac{3(-1)^{n+1}}{11}(3f_{n+1}-f_{n})\\
&=\frac{9}{121}(3f_n-f_{n-1})(2f_n+3f_{n-1})\\
&=\frac{9}{121}(6f_n^2+7f_nf_{n-1}-3f_{n-1}^2)\\
&=\frac{9}{121}(3f_n^2+10f_nf_{n-1}-3(-1)^n).
\end{align*}

If we sum up the exact expressions for $X_n$, $Y_n$, and $Z_n$, then the leading term 
$3(-1)^nf_n^4$ is removed as follows:
\begin{align*}
\Delta_n
&=X_n+Y_n+Z_n\\
&=3(-1)^nf_n^4+3f_n^2+(-1)^n\\
&-3(-1)^nf_n^4-\frac{30}{11}f_n^2+\frac{6}{11}f_nf_{n-1}-\frac{6(-1)^n}{11}\\
&+\frac{9}{121}(3f_n^2+10f_nf_{n-1}-3(-1)^n)\\
&=\frac{60}{121}f_n^2+\frac{156}{121}f_nf_{n-1}+\frac{28(-1)^n}{121}.
\end{align*}
Then $\Delta_n\geq\frac{60}{121}-\frac{28}{121}>0$ for all $n\geq1$.
It completes the proof.
\end{proof}

\subsection{\bf{Lower bound}}\label{subsec-lower}

By Proposition \ref{prp-Delta}, we have
\begin{align}\label{ineq-1}
\frac{1}{g_n}-\frac{1}{g_{n+1}}-\frac{1}{f_n^3}
=\frac{\Delta_n}{g_ng_{n+1}f_n^3}>0
\end{align}
for all positive integers $n$, where $\Delta_n$ is defined as in Proposition \ref{prp-Delta}. By the inequality (\ref{ineq-1}), we have
\begin{align*}
S_n=\sum^\infty_{k=n}\frac{1}{f_k^3}<\sum^\infty_{k=n}
\left(
\frac{1}{g_k}-\frac{1}{g_{k+1}}
\right)=\frac{1}{g_n}.
\end{align*}
It follows that $g_n<S_n^{-1}$ for all positive integers $n$.

\vskip1pc

\subsection{\bf{Upper bound}}

To establish the upper bound, we will show that
\begin{align*}
\frac{1}{g_n+c_n}-\frac{1}{g_{n+1}+c_{n+1}}-\frac{1}{f_n^3}<0.
\end{align*}

Consider $c_n=2/f_n$. 
If we multiply the denominators to define the modified numerator $\Delta_n^*$ by
\begin{align*}
\Delta_n^*
&:=\left[(g_{n+1}-g_n)+(c_{n+1}-c_n)\right]f_n^3-
(g_n+c_n)(g_{n+1}+c_{n+1}).
\end{align*}

Recall that in subsection \ref{subsec-lower},
we computed the formula of 
$\Delta_n=(g_{n+1}-g_n)f_n^3-g_ng_{n+1}$. 
Surprisingly, $\Delta_n^*$ can be decomposed into 
$\Delta_n$ and the remainder terms.
Precisely, we have
\begin{align*}
\Delta_n^*
=\Delta_n+(c_{n+1}-c_n)f_n^3-T_n,
\end{align*}
where  the strictly positive term $T_n$ is defined by
\begin{align*}
T_n=c_ng_{n+1}+c_{n+1}g_n+c_nc_{n+1}.
\end{align*}

By substituting the formula $g_n=h_n+E_n$, we rewrite $T_n$ as
\begin{align*}
T_n=c_n(h_{n+1}+E_{n+1})+c_{n+1}(h_n+E_n)+c_nc_{n+1}.
\end{align*}
Since $c_{n+1}h_n>0$ and $c_nc_{n+1}>0$, 
we have
\begin{align}\label{ineq-Tn}
T_n>\frac{2}{f_n}(f_{n+1}^3-f_n^3)-\frac{2|E_{n+1}|}{f_n}-\frac{2|E_n|}{f_{n+1}}.
\end{align}

Since the inequality $f_{n-1}>\frac{1}{2}f_n$ holds for any $n\geq5$, we have 
\begin{align*}
f_{n+1}^3-f_n^3=3f_n^2f_{n-1}+3f_nf_{n-1}^2+f_{n-1}^3>3f_n^2f_{n-1}>\frac{3}{2}f_n^3.
\end{align*}
Simultaneously, we estimate 
\begin{align*}
|E_n|=\frac{3}{11}(3f_n-f_{n-1})<\frac{3}{11}\left(3f_n-\frac{1}{2}f_n\right)<f_n<f_{n+1}
\end{align*}
and applying the similar argument to obtain
\begin{align*}
|E_{n+1}|<f_{n+1}<2f_n.
\end{align*}
Substituting the above estimates into the inequality (\ref{ineq-Tn}), we have
\begin{align*}
T_n>\frac{2}{f_n}\cdot\frac{3}{2}f_n^3-2\cdot2-2=3f_n^2-6.
\end{align*}
Since the sequence $c_n$ is strictly decreasing, the term 
$(c_{n+1}-c_n)f_n^3$ is strictly negative.
Consequently, 
\begin{align*}
\Delta_n^*<\Delta_n-T_n&<
\frac{60}{121}f_n^2+\frac{156}{121}f_nf_{n-1}+\frac{28}{121}-3f_n^2+6<-f_n^2+7<0
\end{align*}
for all sufficiently large $n$. 
It follows that $\Delta_n^*<0$.

Since the denominator $(g_n+c_n)(g_{n+1}+c_{n+1})f_n^3$ 
is strictly positive,  the inequality $\Delta_n^*<0$ implies that 
\begin{align*}
\frac{1}{g_n+c_n}-\frac{1}{g_{n+1}+c_{n+1}}-\frac{1}{f_n^3}
=\frac{\Delta_n^*}{(g_n+c_n)(g_{n+1}+c_{n+1})f_n^3}<0.
\end{align*}
By the standard telescoping method again, we have 
\begin{align*}
S_n=\sum^\infty_{k=n}\frac{1}{f_k^3}>
\sum^\infty_{k=n}\left(
\frac{1}{g_k+c_k}-\frac{1}{g_{k+1}+c_{k+1}}
\right)=\frac{1}{g_n+c_n}.
\end{align*}
By taking the reciprocal of both sides, we obtain 
$S_n^{-1}<g_n+c_n$ for sufficiently large $n$. 
It completes the proof of Theorem \ref{thm-main}.

\vskip1pc

\section{\bf{Concluding remarks}}

In this paper, we determined the continuous approximation and the exact integer part of the inverse reciprocal sum of the cubes of Fibonacci numbers. 
Unlike the previous cases of $s=1,2$ and $4$, 
the case $s=3$ contains the modulo-5 periodic fractional remainder $w_n$. 

The method in this paper can be generalized to higher powers $s\geq5$.
However, discovering the explicit formula of $g_{s,n}$ will be a significantly more challenging problem, 
as it requires more complicated oscillating terms similar to the process explained in Section 3.
Moreover, our results and methods can be extended to 
generalized Fibonacci numbers \cite{Ramirez14} and other second-order linear recurrence sequences.

\end{document}